\documentclass[12pt]{article}
\usepackage{latexsym}
\usepackage{amsfonts}
\usepackage{amsmath}
\usepackage{epsfig}
\usepackage{amsthm,xspace}
\usepackage{multirow}

\begin{document}
\begin{center}
\vskip 1cm{\LARGE\bf Restricting Dyck Paths and 312-avoiding Permutations} \vskip 1cm \large {Elena Barcucci, Antonio Bernini,\\

 Stefano Bilotta, Renzo Pinzani}\\

\vspace{.2cm}

Dipartimento di Matematica e Informatica ``Ulisse Dini''\\
Universit\`a di Firenze\\

Viale G. B. Morgagni 65,
50134 Firenze, Italy\\

\vspace{.2cm}

\end{center}

\theoremstyle{plain}
\newtheorem{theorem}{Theorem}
\newtheorem{corollary}[theorem]{Corollary}
\newtheorem{lemma}[theorem]{Lemma}
\newtheorem{proposition}[theorem]{Proposition}

\theoremstyle{definition}
\newtheorem{defi}[theorem]{Definition}
\newtheorem{example}[theorem]{Example}
\newtheorem{conjecture}[theorem]{Conjecture}

\theoremstyle{remark}
\newtheorem{remark}[theorem]{Remark}

\begin{abstract}
Dyck paths having height at most $h$ and without valleys at height $h-1$ are combinatorially interpreted by means of 312-avoding permutations with some restrictions on their \emph{left-to-right maxima}. The results are obtained by analyzing a restriction of a well-known bijection between the sets of Dyck paths and 312-avoding permutations.  We also provide a recursive formula enumerating these two structures using ECO method and the theory of production matrices.

As a further result we obtain a family of combinatorial identities involving Catalan numbers. 
\end{abstract}



\section{Introduction}

Dyck paths have been widely used in several combinatorial applications. Here, we only recall their involvement in theory of codes \cite{BBP2,bi}, cryptography \cite{SAB} and partial ordered structures \cite{BCFS}.
Dyck paths enumeration has also received much attention in recent decades. An interesting paper dealing with this matter is the one by E. Deutsch \cite{D} where the author enumerates Dyck paths according to various parameters. 

A subclass of these
paths has been considered thanks to the simple
behavior of the recursive relations describing them and the rational nature of the related generating function. 
More precisely, 
the generating function associated to Dyck paths is algebraic and it is rational when the paths are bounded \cite{BM,BMP}, for example with respect to the height. Kallipoliti et al. \cite{KST} consider Dyck paths having height less or equal to a precise value $k$. Moreover, in the same paper a further restriction is considered:  the authors analyze some characteristics of Dyck paths avoiding valleys at specified height.

In our work we consider Dyck paths having height equal or less than $h$ and having no valleys at height $h-1$. 
We obtain an interesting relation with a subclass of $312$-avoiding permutations (actually, we obtain a bijection) having some constraints on the \emph{left-to-right} maxima.  

The paper structure is the following. In Section \ref{Prelim} some preliminaries on Dyck paths and pattern avoiding permutations are introduced. Here we also recall a well-known bijection between the sets of Dyck paths and $312$-avoiding permutations, we are going to largely use in the whole paper. 
Sections \ref{Section3} and \ref{Section4} are devoted to the generation of the considered Dyck paths (having height equal or less than $h$ and having no valleys at height $h-1$) and the corresponding $312$-avoiding permutations with some restriction on their left to right maxima. The enumerative results are presented in Section \ref{enum}. they provide the generating functions for the above mentioned classes, and a recurrence relation for their enumeration according to their size.

Finally, we conclude the paper proposing some further developments on the present topics.

\section{Preliminaries}\label{Prelim}
A Dyck path is a lattice path in the discrete plane $\mathbb Z^2$ from $(0,0)$ to $(2n,0)$ with up and down steps in $\{(1,1),(1,-1)\}$, never crossing the $x$-axis. The number of up steps in any prefix of a Dyck path is greater or equal to the number of down steps and the total number of steps (the \emph{length} of the path) is $2n$. We denote the set of Dyck paths having length $2n$ (or equivalently semilength $n$) by $\mathcal D_n$.
A Dyck path can be codified by a string over the alphabet $\{U,D\}$, where $U$ and $D$ replace the up and down steps, respectively. The empty Dyck path is denoted by $\varepsilon$.

The height of a Dyck path $P$ is the maximum ordinate reached by one of its steps.
A \emph{valley} of $P$ is an occurrence of the substring $DU$ while a \emph{peak} is an occurrence of the substring $UD$. The height of a valley (peak) is the ordinate reached by $D$ ($U$). 

We denote by $\mathcal D_n^{(h,k)}$ the set of Dyck paths having semilength $n$ and height at most $h$, and avoiding $k-1$ consecutive valleys at height $h-1$. The set of Dyck paths having semilength $n$ with height at most $h$ (without restriction on the number of valleys) is denoted by $\mathcal D_n^{(h)}$. Moreover, $\mathcal D^{(h)}=\displaystyle\sum_{n\geq 0} \mathcal D_n^{(h)}$ and $\mathcal D^{(h,k)}=\displaystyle\sum_{n\geq 0} \mathcal D_n^{(h,k)}$.

The cardinalities of $\mathcal D_n^{(h,k)}$ and $\mathcal D_n^{(h)}$ are indicated by $D_n^{(h,k)}$ and $D_n^{(h)}$, respectively. Finally,
the set $\mathcal D_n$ of unrestricted Dyck paths having semilength $n\geq 0$ is enumerated by the $n$-Catalan number 
$$C_n=\frac{1}{n+1}\binom{2n}{n}\ .$$
  
When $k=2$, the set $\mathcal D_n^{(h,2)}$ represents the set of Dyck paths avoiding valleys at height $h-1$. In the present work we describe a combinatorial interpretation of $\mathcal D_n^{(h,2)}$ in terms of restricted
permutations.

\bigskip
In our context, the above mentioned permutations are related to the notion of pattern avoidance which can be generally described as the absence of a substructure inside a larger structure. In particular, an occurrence of a pattern $\sigma$ in a permutation $\pi=\pi_1 \pi_2 \dots \pi_n$ of length $n$ is a sub-sequence (not necessarily constituted by consecutive entries) of $\pi$ whose entries appear in the same relative order as those in $\sigma$. Otherwise, we say that $\pi$ avoids the pattern $\sigma$, or that $\sigma$ is a forbidden pattern for $\pi$. For example, $\pi=352164$ contains two occurrences of $\sigma=312$ in the sub-sequences $514$ and $524$, while $\pi=34251$ avoids the pattern $\sigma$. The set $\mathcal S_n(312)$ denotes the set of 312-avoiding permutations of length $n$ which are enumerated by the $n$-Catalan number. 

We are going to briefly recall a well-known bijection $\varphi$, useful in the rest of the paper, between the classes $\mathcal D_n$ and $\mathcal S_n(312)$ (see for example \cite{D,K}). Fix a Dyck path $P$
and label its up steps by enumerating them from left to right (so that the $\ell$-th up
step is labelled $\ell$). Next assign to each down step the same label of the up step it
corresponds to. Now consider the permutation  whose entries are constituted by the
labels of the down steps read from left to right. Such a permutation $\pi=\varphi(P)$ is easily seen
to be 312-avoiding. As far as the inverse map $\varphi^{-1}$ is concerned, once fixed a 312-avoiding permutation $\pi=\pi_1\pi_2 \ldots \pi_n$ we can consider its factorization in terms of descending sub-sequences whose first elements coincide with the \emph{left-to-right maxima} of $\pi$. A left-to-right maximum (l.r.M for short) is an element $\pi_i$ which is greater 
than all the elements to its left, i.e., greater than all  $\pi_j$ with $j<i$. Denoting $\pi_{i_1}, \pi_{i_2},\ldots, \pi_{i_\ell}$ the left-to-right maxima of $\pi$, the corresponding Dyck path $P=\varphi^{-1}(\pi)$ is obtained as follows: 
\begin{itemize}
    \item write as many $U$'s as $\pi_{i_1} (=\pi_1)$ followed by as many $D$'s as the cardinality of the first descending sub-sequence headed by $\pi_{i_1}$;
    \item for each $j=2,\ldots,\ell$, add as many $U$'s as $\pi_{i_j} - \pi_{i_{j-1}}$ followed by as many $D$'s as the cardinality of the sub-sequence headed by $\pi_{i_j}$.
\end{itemize}

Two easy properties of l.r.M of $\pi\in S_n(312)$, and the corresponding steps in $P=\varphi^{-1}(\pi)$ are summarized in the following: 

\begin{proposition}\label{property}
Let $P$ be a Dyck path in $\mathcal D_n$ and $\pi=\varphi(P)=\pi_1 \ldots \pi_n$ be the associated permutation in $S_n(312)$.
Each entry $\pi_{i_j}$ corresponding to the first down step of a sub-sequence of consecutive down steps in $P$ is a left to right maximum. Moreover, 
$\pi_{i_j} - i_j$ is the height reached by the descending step in $P$ corresponding to $\pi_{i_j}$.

\end{proposition}

\section{A generating algorithm}\label{Section3}
The set $\mathcal D^{(h,2)}$ can be exhaustively generated by means an ECO operator \cite{BDPP1} which allows to construct all the paths of a certain length $n+1$ (the size of the combinatorial objects) staring from the ones of size $n$.

To this aim, consider a Dyck path $P \in \mathcal D_n^{(h,2)}$ which, obviously, starts with $t \leq h$ up steps $U$. We mark these steps factorizing the path $P$ as $P=U_1 U_2 \cdots U_t D P'$, where $P'$ is a suitable Dyck suffix of length $n-t-1$.
The idea is to consider some \emph{sites} in $P \in \mathcal D_{n}^{(h,2)}$ where an insertion of the factor $\mathbf{UD}$ is allowed in order to obtain paths in $\mathcal D_{n+1}^{(h,2)}$ from $P$ (so that the sites are called \emph{active sites}).    

Thus, we define an operator $\vartheta$ for the class $\mathcal D_n^{(h,2)}$ as follows:
\begin{itemize}
\item[-] if $P=U_1 U_2 \cdots U_{t-1}U_t D P' \in \mathcal D_n^{(h,2)}$, with $t < h $, then 
$$
\begin{array}{ll}
\vartheta(P)=\{&\hspace{-.35cm}\mathbf{UD} U_1 U_2 \cdots U_{t-1}U_t D P',\\
&\hspace{-.35cm}U_1\mathbf{UD} U_2 \cdots U_{t-1}U_t D P',\\
& \cdots\\
&\hspace{-.35cm}U_1 U_2 \cdots U_{t-1} \mathbf{UD} U_t D  P',\\
&\hspace{-.35cm}U_1 U_2 \cdots U_{t-1} U_t \mathbf{UD} D P'\ \}\ ;
\end{array}
$$
\item[-] if $P=U_1 U_2 \cdots U_{t-1}U_t D P' \in \mathcal D_n^{(h,2)}$, with $t=h$, then  
$$
\begin{array}{ll}
\vartheta(P)=\{&\hspace{-.35cm}\mathbf{UD} U_1 U_2 \cdots U_{t-1}U_t D P',\\
&\hspace{-.35cm}U_1\mathbf{UD} U_2 \cdots U_{t-1} U_t D P',\\
& \cdots\\
&\hspace{-.35cm}U_1 U_2 \cdots \mathbf{UD} U_{t-1} U_t D  P'\}\ .
\end{array}
$$
\end{itemize}

We note that the insertion of $\mathbf{UD}$ may create a valley in the paths of $\vartheta(P)$. In particular,
\begin{itemize}
    \item the insertion of $\mathbf{UD}$ before the step $U_j$, with $j=1,2,\ldots,t-1$, gives the occurrence of the valley $\mathbf{D}U_j$ having height $j-1<h-1$ in any case;
    \item the insertion of $\mathbf{UD}$ before the step $U_t$ in the case $t<h$ gives the occurrence of the valley $\mathbf{D}U_t$ having height equal to $t-1<h-1$;
    \item the insertion of $\mathbf{UD}$ after the step $U_t$ in the case $t<h$ does note give the occurrence of a valley (since the next step is again a $D$ step).
\end{itemize}

\noindent
In other words, the valley possibly generated by the insertion of $\mathbf{UD}$ has height less than $h-1$, therefore we have:
\begin{proposition}
    If $x \in \vartheta(P)$, with $P \in \mathcal D_n^{(h,2)}$,  then $x \in \mathcal D_{n+1}^{(h,2)}$.
\end{proposition}

\bigskip
In the spirit of ECO method, we have to prove the following proposition.
\begin{proposition}
    The operator $\vartheta$ is an ECO operator. 
\end{proposition}

\begin{proof}
The proof consists in the following steps:
\begin{itemize}
\item[i)] If $x,y \in \mathcal D_n^{(h,2)}$ with $x \neq y$, then $\vartheta(x) \cap \vartheta(y)=\emptyset$.

\item[ii)] If $x \in \mathcal D_{n+1}^{(h,2)}$ then $\exists\ y \in \mathcal D_n^{(h,2)}$ such that $x \in \vartheta(y)$.

\end{itemize}
For case i), we suppose that exists a path $P$ such that $P \in \vartheta (x)$ and $P \in \vartheta(y)$, with $x \neq y$.
From the description of the operator $\vartheta$ it is easy to realize that the first peak of $P$ is precisely generated by the insertion of the factor $\mathbf{UD}$. By removing such a peak  from $P$, we obtain an unique path. Thus, we would have $x=y$, against the hypothesis.      

\bigskip
\noindent
For case ii), if $x\in D_{n+1}^{(h,2)}$, then $x=U^j\mathbf{UD}T'$, with $j=0,1,\ldots, h-1$, where $T'$ is a suitable Dyck suffix of suitable length.
Then, the path $y=U^jT'$ starts with at most $h$ up steps $U$ so that $x\in 
D_n^{(h,2)}$. Clearly, it is $x\in\vartheta(y)$ since $y$ is obtained by the insertion of $\mathbf{UD}$ in $x$.

\end{proof}

A generating algorithm can be naturally described by means of the concept of \emph{succession rule}. Such a concept was introduced by Chung et al. \cite{CGHK} to study reduced Baxter permutations. Recently this technique has been successfully applied to other
combinatorial objects \cite{Bilotta2013157,Bilotta201910} and it has been recognized as an extremely useful tool for the ECO method \cite{BDPP1}. In all these cases there is a common approach
to the examined enumeration problem: a generating tree is associated to certain
combinatorial class, according to some enumerative parameters, in such
a way that the number of nodes appearing on level $n$ of the tree gives the
number of $n$-sized objects in the class. 

A succession rule is a formal system constituted by an \emph{axiom} $(a)$ and some \emph{productions} (possibly only one) having the form

$$
(k)\rightsquigarrow \left(e_1(k)\right)\left(e_2(k)\right)\ldots
\left(e_k(k)\right)\ ,
$$

\noindent
so that a succession rule $\Omega$ is often denoted by

$$
\Omega:
\left\{
\begin{array}{l}
     (a)\\
     \\
     (k)\rightsquigarrow \left(e_1(k)\right)\left(e_2(k)\right)\ldots
\left(e_k(k)\right) \ .\\
\end{array}
\right.
$$

\medskip
\noindent
The symbols $(a)$,$(k)$ and $e_i(k)$ are called \emph{labels} (their values are positive integers), and play a crucial role when the the succession rule $\Omega$ is represented by a \emph{generating tree}. This is a rooted tree whose nodes are the labels of $\Omega$. More precisely, the root is labelled with $(a)$ and each node having label $(k)$ has $k$ children having labels
$e_1(k),e_2(k),\ldots,e_k(k)$, according to the productions in $\Omega$.

In our case the generating algorithm for $\mathcal{D}^{(h,2)}$ is performed by the operator $\vartheta$ and from its definition is easy to realize that:
\begin{itemize}
    \item the empty path $\varepsilon$ can be labelled with the axiom $(1)$ having production $(1)\rightsquigarrow (2)$: the path $\varepsilon$ generates the path $UD$, having in turns label $(2)$;
    \item every other path $P$ can have label $(2)$, $(3)$, \ldots or $(h)$ depending on the number $t$ of its starting up steps $U$. More precisely, if $1 \leq t \leq h-1$ then $P$ is labelled $(t+1)$. Otherwise, if $t=h$, then $P$ is labelled $(h-1)$. 
\end{itemize}

In order to write the productions of the labels $(k)$ of $P$, with $k=2,3,\ldots,h$ we observe that:
\begin{itemize}
    \item if $k<h$ then the $k$ paths in $\vartheta(P)$ start, respectively, with $1,2,\ldots\ \mbox{or}\ k$ up steps, so that, in turns, they are labelled $(2),(3),\ldots (k+1)$. Then we can write the production $$(k)\leadsto(2)(3)\cdots(k)(k+1)
    , \ 2\leq k <h.$$
    \item if $k=h$ then the $k$ paths in $\vartheta(P)$ start, respectively, with $1,2,\ldots\ \mbox{or}\ h$ up steps. Since the path having $h$ starting up steps is labelled $(h-1)$, then we can write the production $$(h)\leadsto(2)(3)\cdots(h-1)^2(h).$$
    The two paths having label $(h-1)$ are precisely the one starting with $h$ up steps and the one starting with $h-2$ up steps.
\end{itemize}

Finally, the generating algorithm for $\mathcal D^{(h,2)}$ can be described by the succession rule (for $h\geq3$) as follows:

\begin{equation}\label{rule_h2}
\Omega_h:
\left\{
\begin{array}{ccl}
(1)&&\\
(1)&\leadsto&(2)\\
(k)&\leadsto&(2)(3)\cdots(k)(k+1), \quad\quad 2\leq k<h\\
(h)&\leadsto&(2)(3)\cdots(h-1)^2(h)
\end{array}
\right.
\end{equation}

\section{The bjection with a subset of 312-avoiding permutations}\label{Section4}

Let $\mathcal S_n^{(h)}(312)\subseteq S_n(312)$ be the subset of permutations $\pi\in S_n(312)$ such that $\pi_{i_j} - i_j \leq h-1$, for each l.r.M. $\pi_{i_j}$ of $\pi$.
The reader can easily check that the restriction
$\varphi_{\big|D_n^{(h)}}$
of $\varphi$ to the set $D_n^{(h)}$ is a bijection between $D_n^{(h)}$ and 
$S_n^{(h)}(312)$ (using Proposition \ref{property}).

We now consider the paths in $D_n^{(h,2)}$ and characterize the 
corresponding permutations via the restriction of $\varphi$ to this set. 
The following proposition holds.

\begin{proposition}
Let $P$ be a Dyck path in $\mathcal D_n$. Then, $P \in \mathcal D_n^{(h,2)}$ if and only if in the corresponding permutation $\pi=\varphi(P)$ there is no left-to-right maximum $\pi_{i_j}$ such that 
\begin{enumerate}
    \item $\pi_{i_j} - i_j =h-1$ and
    \item $\pi_{i_{j+1}}=\pi_{i_j}+1$.
\end{enumerate}
\end{proposition}

\begin{proof}
Suppose that $\pi=\varphi(P)$ has no a left-to-right maximum $\pi_{i_j}$ such that $\pi_{i_j} - i_j =h-1$ and $\pi_{i_{j+1}}=\pi_{i_j}+1$.
Let $P=\varphi^{-1}(\pi)$ be the corresponding path. We have to prove that $P\in \mathcal D_n^{(h,2}$.
\begin{itemize}
\item If $P$ has height less than $h$, then, surely, $P\in \mathcal D_n^{(h,2)}$ and the proof is completed.

\item Suppose that $P$ has height equal to $h$ and suppose, ad absurdum, that $P\notin \mathcal D_n^{(h,2)}$. Therefore, there exists a valley having height $h-1$. Thus, $P$ can be written as $P'U_iD_iU_{i+1}D_{i+1}P''$,
where $P'$ and $P''$ are, respectively, a Dyck prefix and a Dyck suffix having height $h-1$.
Considering the permutation $\pi=\varphi(P)=\pi_1 \ldots \pi_i \pi_{i+1} \ldots \pi_n$ (where we highlighted the entries $\pi_i$ and $\pi_{i+1}$ corresponding to the steps $D_i$ and $D_{i+1}$), thanks to Proposition \ref{property}, it is possible to observe that the elements $\pi_i$ and $\pi_{i+1}$ associated to $U_i$ and $U_{i+1}$, respectively, are l.r.M. in $\pi$. Again from Proposition \ref{property}, we have $\pi_i-i=h-1$ and $\pi_{i+1}-(i+1)=h-1$ and, by substitution, it is $\pi_{i+1}=\pi_i +1$, against the hypothesis. Thus, $P\in \mathcal D_n{(h,2)}$.

\end{itemize}

On the other side, suppose that $P \in \mathcal D_n^{(h,2)}$ and suppose, ad absurdum, that $\pi=\varphi(P)=\pi_1 \ldots \pi_i \pi_{i+1} \ldots \pi_n \in \mathcal S_n^{(h)}(312)$ has a left to right maximum $\pi_i$ with  $\pi_{i+1}=\pi_i+1$ and $\pi_i - i = h-1$. Then, it is $\pi=\pi_1 \ldots \pi_i (\pi_i+1) \ldots \pi_n$. Since $\pi_i < \pi_{i+1}$ and $\pi$ is a 312-avoiding permutation, then there is not $\pi_l > \pi_i$ with $l<i$. Thus, both $\pi_i$ and $\pi_{i+1}$ are l.r.M. in $\pi$. From Proposition \ref{property}, the quantities $\pi_i - i$ and $\pi_{i+1} - (i+1)$ are the heights reached by the corresponding descending steps in $P$. Moreover, from the two hypotheses $\pi_i - i =h-1$ and $\pi_{i+1}=\pi_i+1$, we deduce $\pi_{i+1} - (i+1)=\pi_i+1 - (i+1)= h-1$. Thus, $P=\varphi^{-1}(\pi)$ can be factorized as $P=P'U_iD_iU_{i+1}D_{i+1}P''$ showing that $P$ admits a valley having height $h-1$, against the hypothesis $P\in \mathcal D_n^{(h,2)}$.   

\end{proof}
\bigskip

The permutations corresponding to the paths in $\mathcal D_n^{(h,2)}$ are denoted by $\mathcal S_n^{(h,2)}(312)$. By means of the above proposition, we proved the following one.

\begin{proposition}\label{equiClass}
There exists a bijection between the classes  $\mathcal S_n^{(h,2)}(312)$ and $\mathcal D_n^{(h,2)}$, which is the restriction $\varphi_{\big|D_n^{(h,2)}}$.
\end{proposition}

\noindent From Proposition \ref{equiClass}, a generating algorithm for the class $\mathcal S_n^{(h,2)}(312)$ according to the succession rule $\Omega_h$ can be obtained. A combinatorial interpretation of $\Omega_h$ in terms of permutations is then desired. 

First of all we note that, if $\pi=\pi_1 
\dots \pi_n \in \mathcal S_n^{(h,2)}(312)$, then $\pi_1 \leq h$. 
After that, we have to find an interpretation of the parameters 
comparing in the rule $\Omega_h$. The axiom $(1)$ at level 0 can 
be associated to the \emph{empty} permutation and its production 
labelled with $(2)$ can be associated to the permutation $1$. The 
parameter $(k)$ at level $n$ in the rule $\Omega_h$ admits the 
following interpretation according to the value of $\pi_1$ in 
$\pi \in S_n^{(h,2)}(312)$:

\begin{equation}\label{k_rule}
(k)=
\begin{cases}
\pi_1 + 1  & \text{if } \pi_1 \neq h\ ;\\
\pi_1 - 1  & \text{if } \pi_1 = h .
\end{cases}
\end{equation}

More precisely, if $\pi_1 < h$, a permutation $\pi=\pi_1 \dots \pi_n \in \mathcal S_n^{(h,2)}(312)$ at level $n$, produces $k=\pi_1+1$ sons at level $n+1$ by inserting the element $\ell$, with $\ell=1,2,\ldots, \pi_1 +1$, before $\pi_1$ and rescaling the sequence $\ell\pi$ in order to obtain a permutation $\pi' \in \mathcal S_{n+1}^{(h,2)}(312)$ (for the sake of clearness, each entry $\pi_i$ of $\pi$ equal or greater than $\ell$ is increased by $1$ in order to obtain $\pi'$).

Otherwise, when $\pi_1 = h$, a given permutation $\pi=\pi_1 \dots \pi_n \in \mathcal S_n^{(h,2)}(312)$ at level $n$, produces $k=\pi_1-1=h-1$ sons at level $n+1$ by inserting the element $\ell$, with $\ell=1,2,\ldots, h-1$, before $\pi$. Analogously, $\pi' \in \mathcal S_{n+1}^{(h,2)}(312)$ is obtained by rescaling the sequence $\ell \pi$, for each $\ell$.  

\bigskip
As an example, fixed $h=3$, the succession rule for $\mathcal S_n^{(3,2)}(312)$, or equivalently for $\mathcal D_n^{(3,2)}$, is as follows:
\begin{equation}\label{rule_h_3}
\Omega_3:
\left\{
\begin{array}{ccl}
(1)&&\\
(1)&\leadsto&(2)\\
(2)&\leadsto&(2)(3)\\
(3)&\leadsto&(2)(2)(3)
\end{array}
\right.
\end{equation}

In Figure \ref{fig1} a graphical representation of the first levels of $\Omega_3$ is shown in terms of permutations in $\mathcal S_n^{(3,2)}$.

\begin{figure}[!h]
\centering
\includegraphics[width=1.29\linewidth]{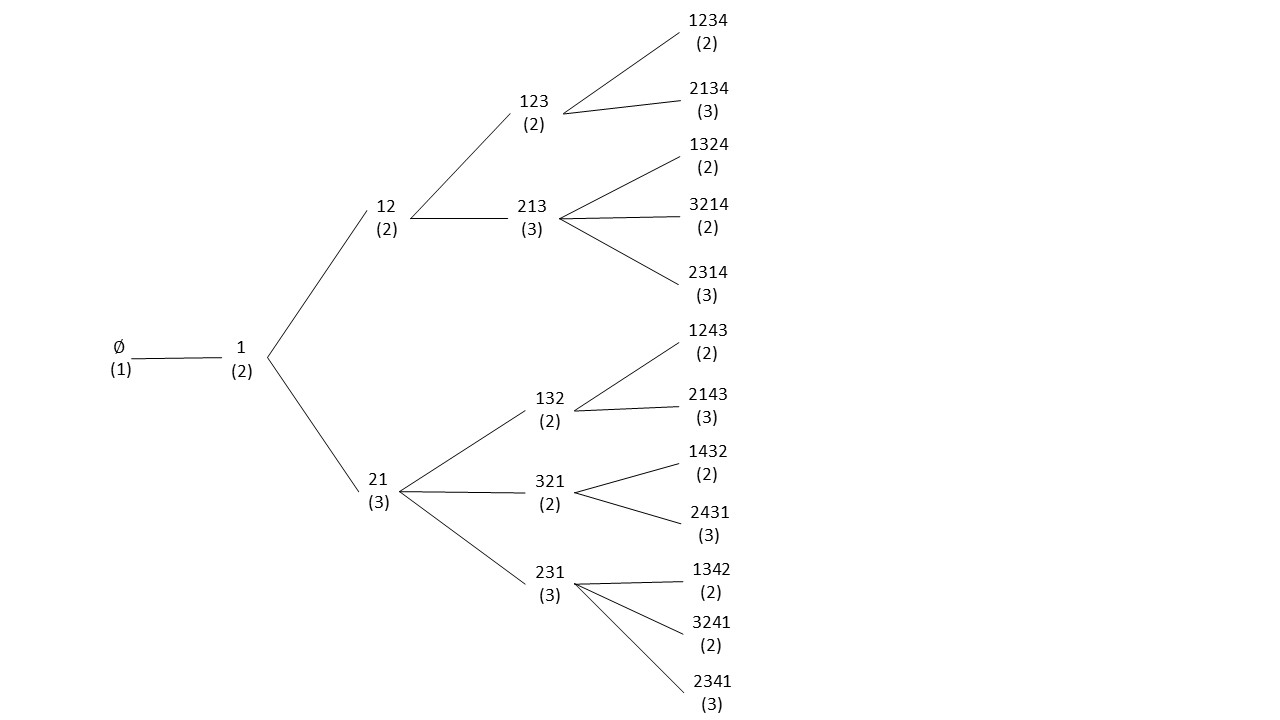}
\caption{Graphical representation of the generating tree associated with $\mathcal S_n^{(3,2)}$ where the label associated to each permutation is shown in brackets.}
\label{fig1}
\end{figure}

\section{Enumeration} \label{enum}

The case $h=2$ is not included in the general formula (\ref{rule_h2}) for the succession rules. However, it is easy to see that in this case it is

\begin{equation}\label{rule_h_2}
\Omega_2:
\left\{
\begin{array}{ccl}
(1)&&\\
(1)&\leadsto&(2)\\
(2)&\leadsto&(1)(2)\\
\end{array}
\right.
\end{equation}

\noindent
The succession rule (\ref{rule_h_2}) defines the Fibonacci numbers.

According to the theory developed by Deutsch et al. \cite{DFR}, the production matrix $P_2$ associated to $\Omega_2$ is

\begin{equation}
P_2=
\left(
\begin{matrix}
0&1\\
1&1\\
\end{matrix}
\right)
\end{equation}

\noindent
and, for $h\geq 3$, 


\begin{equation}
P_h=
\left(
\begin{matrix}
0&u^t\\
0&P_{h-1}+eu^t\\
\end{matrix}
\right)
\end{equation}

\bigskip
\noindent
where $u^t$ is the row vector $(1\ 0\ 0\ldots)$ and $e$ is the column vector $(1\ 1\ 1\ \ldots)^t$, of appropriate size and for what the generating function $f_h(x)$ of the sequence corresponding to $\Omega_h$ is concerned, we have \cite{DFR}, for $h\geq 2$,

\begin{equation}\label{f_h}
f_h(x)=\frac{1}{1-xf_{h-1}(x)}\ .
\end{equation}

When $h=1$, clearly the unique paths in $\mathcal D_n^{(1,2)}$ are the empty path $\varepsilon$ and $UD$, so that the sequence $(D_n^{(1,2)})_{n\geq 0}$ is
$\{1,1,0,0,\ldots\}$, whose generating function is 
 $f_1(x)=1+x$ which is rational. Thanks to (\ref{f_h}) it is possible to deduce that also $f_h(x)$ with $h\geq 2$ is rational, too. Therefore, we can consider its general form as follows:

\begin{equation}\label{f_hpq}
f_h(x)=\frac{p_h(x)}{q_h(x)}\ ,
\end{equation}

\noindent
where $p_h(x)$ and $q_h(x)$ are polynomials with suitable degrees.  From (\ref{f_h}) and (\ref{f_hpq}) we obtain

\begin{equation}\label{p_hq_h}
\begin{aligned}
p_h(x) & = q_{h-1}(x)\\
q_h(x) &= q_{h-1}(x)-xq_{h-2}(x)\ .\\
\end{aligned}
\end{equation}

\bigskip
\noindent
Since the degree of the polynomial $q_h(x)$ is $\left\lceil{\frac{h+1}{2}}\right\rceil$ (it can be easily seen by induction), we can assume

$$
q_h(x)=a_{h,0}-a_{h,1}x-a_{h,2}x^2-\ldots-a_{h,j}x^j
\quad 
\mbox{with} \quad j=\left\lceil{\frac{h+1}{2}}\right\rceil .
$$

\noindent
Clearly, it is $a_{h,j}=0$ if $j>\left\lceil{\frac{h+1}{2}}\right\rceil$.

\bigskip
\noindent
As $a_{1,0}=1$, thanks to (\ref{p_hq_h}) we have $a_{h,0}=a_{h-1}$,
and $$a_{h,0}=1\ \ \mbox{for each}\ \  h\geq 1\ .$$ Moreover,

\begin{equation}\label{q_h_prima}
q_h(x)=1-a_{h,1}x-a_{h,2}x^2-\ldots-a_{h,j}x^j
\quad \mbox{with} \quad j=\left\lceil{\frac{h+1}{2}}\right\rceil .
\end{equation}

\bigskip
\noindent
Using the expression for $q_h(x)$ in (\ref{p_hq_h}), we obtain

\begin{equation}\label{q_h_seconda}
\begin{aligned}
q_h(x)= & \ 1-a_{h-1,1}x-a_{h-1,2}x^2-\ldots-a_{h-1,j-1}x^{j-1}\\
 & -x\left(
 1-a_{h-2,1}x-a_{h-2,2}x^2-\ldots-a_{h-2,j-2}x^{j-2}
 \right) .\\
\end{aligned}
\end{equation}

\bigskip
\noindent
For the identity theorem for polynomials, comparing formulas  (\ref{q_h_prima}) and (\ref{q_h_seconda}) for $q_h(x)$, it is

\begin{equation}\label{a_hj}
a_{h,j}=
\begin{cases}
a_{h-1,1}+1  & \text{for } j=1\\
a_{h-1,j} - a_{h-2,j-1} & \text{for } j=2,3,\ldots,\left\lceil{\frac{h+1}{2}}\right\rceil .
\end{cases}
\end{equation}

In Table \ref{rn} we list the first numbers of the coefficients $a_{h,j}$ for some fixed values of $h\geq 1$ and $j \geq 1$. On the diagonals, it is possible to observe a similarity with the $A112467$ sequence in The On-line Encyclopedia of Integer Sequences \cite{S}.

\begin{table}[h!]
\begin{center}
{\footnotesize
\begin{tabular}{|c|rrrrrrrr|}
\hline
{$h$} / {$j$} & 1 & 2 & 3 & 4 & 5 & 6 & 7 & 8\\
\hline
1  & 0 & 0 & 0 & 0 & 0 & 0 & 0 & 0\\
2  & 1 & 1 & 0 & 0 & 0 & 0 & 0 & 0\\
3  & 2 & 1 & 0 & 0 & 0 & 0 & 0 & 0\\
4  & 3 & 0 & -1 & 0 & 0 & 0 & 0 & 0\\
5  & 4 & -2 & -2 & 0 & 0 & 0 & 0 & 0\\
6  & 5 & -5 & -2 & 1 & 0 & 0 & 0 & 0\\
7  & 6 & -9 & 0 & 3 & 0 & 0 & 0 & 0\\
8  & 7 & -14 & 5 & 5 & -1 & 0 & 0 & 0\\
9  & 8 & -20 & 14 & 5 & -4 & 0 & 0 & 0\\
10  & 9 & -27 & 28 & 0 & -9 & 1 & 0 & 0\\
11  & 10 & -35 & 48 & -14 & -14 & 5 & 0 & 0\\
12  & 11 & -44 & 75 & -42 & -14 & 14 & -1 & 0\\
13  & 12 & -54 & 110 & -90 & 0 & 28 & -6 & 0\\
14  & 13 & -65 & 154 & -165 & 42 & 42 & -20 & 1\\

\hline\end{tabular}
}
\caption{The coefficients $a_{h,j}$ for some fixed values of $h$ and $j$.} \label{rn}
\end{center}
\end{table}

\bigskip
\noindent
We have an explicit formula for the coefficients $a_{h,j}$ thanks to the following proposition.

\begin{proposition}
For $h\geq 2$ and for $j=1,2,\ldots,\left\lceil{\frac{h+1}{2}}\right\rceil$ we have:
\begin{equation}\label{a_hj_binom}
a_{h,j}=\frac{3j-h-2}{j}\binom{h-j+1}{j-1}(-1)^j
\end{equation}
\end{proposition}

\begin{proof}\quad We proceed by induction on $h$. For $h=2$, it is $j=1,2$, and expression (\ref{a_hj_binom}) gives $a_{2,1}=1$ and $a_{2,2}=1$, agreeing with the expression for $f_2(x)=\frac{1}{1-x-x^2}$ derived from (\ref{f_h}) and $f_1(x)=1+x$.

For $h>2$, we first analyze the case $j=1$. Using $a_{h,1}=a_{h-1,1}+1$ from (\ref{a_hj}) and the inductive hypothesis, we have

$$
a_{h,1} = \ a_{h-1,1}+1=(2-h)(-1)^1+1=h-1\\
$$

\noindent
which matches the value of $a_{h,1}$ returned by (\ref{a_hj_binom}).

For $j>2$, we use $a_{h,j}=a_{h-1,j}-a_{h-2,j-1}$ from (\ref{a_hj}) and, again, the inductive hypothesis. We get

$$
\begin{aligned}
a_{h,j} &= \ a_{h-1,j}-a_{h-2,j-1}\\
&\\ 
&=\frac{3j-h-1}{j}\binom{h-j}{j-1}(-1)^{j}-\frac{3j-h-3}{j-1}\binom{h-j}{j-2}(-1)^{j-1}\\
&\\
&=\frac{3j-h-1}{j}\binom{h-j}{j-1}(-1)^{j}+\frac{3j-h-3}{j-1}\binom{h-j}{j-2}(-1)^{j}\ .
\end{aligned}
$$

\noindent
Expanding the binomial coefficients and with some manipulations, it is

$$
a_{h,j}=\frac{(-1)^j(h-j+1)!(3j-h-2)}{j(j-1)!(h-2j+2)!}=\frac{3j-h-2}{j}\binom{h-j+1}{j-1}(-1)^j\ ,
$$

\bigskip
\noindent
as required. The proof is completed.

\end{proof}

In the sequel, we are going to evaluate a recurrence relation for the terms $D_n^{(h,2)}$ involving the series expansion at $x=0$ of the generating function
$$f_h(x)=\frac{p_h(x)}{q_h(x)}=\frac{q_{h-1}(x)}{q_h(x)}=\sum_{n\geq 0}D_n^{(h,2)}x^n\ .$$

The expression for $q_h(x)$ becomes
\begin{equation}\label{q_hx_definitiva}
q_h(x)=1-\sum_{j=1}^{\lceil{\frac{h+1}{2}}\rceil}\frac{3j-h-2}{j}\binom{h-j+1}{j-1}(-1)^jx^j
\end{equation}

Thus, we obtain

\begin{equation}\label{gen_fun_pq_definitiva}
f_h(x)=\frac{1-\displaystyle\sum_{j=1}^{\lceil{\frac{h}{2}}\rceil}\frac{3j-h-1}{j}\binom{h-j}{j-1}(-1)^jx^j}{1-\displaystyle\sum_{j=1}^{\lceil{\frac{h+1}{2}}\rceil}\frac{3j-h-2}{j}\binom{h-j+1}{j-1}(-1)^j x^j}\ 
\end{equation}

\bigskip\noindent
and
  
\begin{align*}
 &\left( {1-\displaystyle\sum_{j=1}^{\lceil{\frac{h+1}{2}}\rceil}\frac{3j-h-2}{j}\binom{h-j+1}{j-1}(-1)^j x^j}\right)
  \left(
  \sum_{n\geq 0}D_n^{(h,2)}x^n
  \right)
  =\\
  &=1-\displaystyle\sum_{j=1}^{\lceil{\frac{h}{2}}\rceil}\frac{3j-h-1}{j}\binom{h-j}{j-1}(-1)^jx^j    
\end{align*}

\bigskip\noindent
Sorting the first part according to the increasing powers of $x$ we have

\begin{align*}
	&\sum_{n\geq 0}
	\left(
	D_n^{(h,2)}-
	\sum_{j=1}^{\lceil{\frac{h+1}{2}}\rceil}
	D_{n-j}^{(h,2)}
	\frac{3j-h-2}{j}\binom{h-j+1}{j-1}(-1)^j
	\right)
	x^n
	=\\
	&=1-\displaystyle\sum_{j=1}^{\lceil{\frac{h}{2}}\rceil}\frac{3j-h-1}{j}\binom{h-j}{j-1}(-1)^jx^j
\end{align*}

\bigskip\noindent
where $D_{\ell}^{(h,2)}=0$ whenever $\ell \leq 0$.

\medskip
For the identity theorem for polynomials we can deduce the desired recurrence relation

\scriptsize
\begin{equation*}
D_n^{(h,2)}
=
\begin{cases}
1&\text{for } n=0\ ;\\
&\\
\displaystyle\sum_{j=1}^{\lceil{\frac{h+1}{2}}\rceil}
D_{n-j}^{(h,2)}
\frac{3j-h-2}{j}\binom{h-j+1}{j-1}(-1)^{j}
-\frac{3n-h-1}{n}
\binom{h-n}{n-1}
(-1)^{n}&\text{for } n\geq 1\ .
\end{cases}
\end{equation*}

\normalsize
A very interesting note arises when, once $h$ is fixed, we ask  for the number $D_n^{(h,2)}$ of Dyck paths having semilength $n\leq h$. Clearly, in this case, it is $D_n^{(h,2)}=C_n$ since 
all the Dyck paths of a certain semilegth $n\leq h$ have height at most equal to $n$. Thanks to the above argument it is 
possible to derive interesting relations involving Catalan 
numbers. Indeed, for the above remark, posing $h=n+\alpha,$
we can write
$
D_n^{(n+\alpha,2)}=C_n
$,
where $\alpha \geq 0$ is integer. Then, it is possible to deduce the combinatorial identity involving Catalan numbers as follows:
\footnotesize
\begin{equation}\label{Catalan}
C_n=\displaystyle\sum_{j=1}^{n}
C_{n-j}
\frac{3j-n-\alpha-2}{j}\binom{n+\alpha-j+1}{j-1}(-1)^{j}
-\frac{2n-\alpha-1}{n}
\binom{\alpha}{n-1}
(-1)^{n} \ .
\end{equation}
\normalsize

\section{Further developments}
In this paper we analyzed the case $k=2$ leading to bounded Dyck paths avoiding valleys at given height (i.e., $h-1$) corresponding to the permutations in $\mathcal S_n^{(h,2)}(312)$. An interesting generalization could concern  the cases $k>2$ in order to investigate what are the arising constraints on the subclasses of $312$-avoiding permutations. The number $k-1$ of consecutive valleys allowed at height $h-1$ clearly affects the value and position of the l.r.M., as we have seen in the $k=2$ case. For values of $k$ larger than $2$, the permutations probably have a structure that can be described in terms of a suitable block decomposition.

The above combinatorial identity (\ref{Catalan}) is obtained by means of a purely combinatorial consideration. By virtue of this, similar relations are expected to arise even in cases $k > 2$.
It might then be possible to derive a family of combinatorial identities as $k$ varies. 

Another further line of research could consider the possibility to list the paths of $\mathcal D_n^{(h,2)}$  in a Gray code sense using the tools developed by Barcucci, Bernini et al. \cite{BBP2,BBPSV2,BBPSV,BBPV1,BBPV2}. As mentioned in Section \ref{Prelim}, these paths can be encoded by strings on the alphabet $\{U,D\}$, so the problem of defining a Gray code could be addressed by starting from the techniques developed by Vajnovszki et al. \cite{VW}.

Moreover, the considered Dyck paths could be used for the construction of a strong non-overlapping code proposed by Barcucci et al. \cite{BBP5}.

\end{document}